\documentclass[12pt]{article}
\usepackage{color,latexsym,amsfonts,amssymb,amsmath,mathbbol,bm}
\usepackage{amsthm}
\usepackage{verbatim}
\usepackage{mathrsfs}
\theoremstyle{definition}

\theoremstyle{definition}
\newtheorem*{rem}{Remark}

\newtheorem{theorem}{Theorem}[section]
\newtheorem{lem}[theorem]{Lemma}
\newtheorem{corollary}[theorem]{Corollary}

\theoremstyle{definition}
\newtheorem{example}[theorem]{Example}

\newcommand{\A}{\mathbb{A}}
\newcommand{\B}{\mathbb{B}}
\newcommand{\C}{\mathbb{C}}
\newcommand{\E}{\mathbb{E}}
\newcommand{\Z}{\mathbb{Z}}

\def\Pr{\mathbb{P}}

\newcommand{\Pn}{{\rm Poisson}}
\newcommand{\CP}{{\rm CP}}
\newcommand{\cprv}{{\cal C}}
\newcommand{\NB}{{\rm NB}}

\def\L{\lambda}

\def\FD{FD_{n,m}}
\def\FDo{FD_{n,1}}
\def\Ref#1{(\ref{#1})}
\parskip = \baselineskip
\allowdisplaybreaks
\newcommand{\eqs}{\begin{eqnarray*}}
\newcommand{\ens}{\end{eqnarray*}}

\newcommand{\beas}{\begin{eqnarray*}}
\newcommand{\enas}{\end{eqnarray*}}

\newcommand{\eqa}{\begin{eqnarray}}
\newcommand{\ena}{\end{eqnarray}}
\newcommand{\eq}{\begin{equation}}
\newcommand{\en}{\end{equation}}

\newcommand{\cF}{{\cal F}}
\newcommand{\cI}{{\cal I}}
\newcommand{\cL}{{\cal L}}

\newcommand{\cf}{{\cal F}}

\def\ignore#1{}

\def\Po{{\rm Pn}}

\def\L{\lambda}
\def\resm{{(m)}}
\def\resmone{{(m-1)}}
\def\resone{{(1)}}

\makeatletter
\renewcommand\theequation{\thesection.\@arabic\c@equation}
\makeatother
\usepackage{color,graphicx,mathbbol, enumerate}

\newcommand\bp{\boldsymbol{\pi}}
\newcommand\fn{\mathcal{N}}

\def\ind{\bm{1}}  
\def\E{\mathbb{E}} 

\def\Pr{\mathbb{P}} 
\def\Z{\mathbb{Z}} 
\def\unt{{u_n^{(\tau)}}}

\def\Ref#1{(\ref{#1})}

\newcommand{\Nb}{\text{NB}}

\numberwithin{equation}{section}  
\begin{document}

\title{Fragility distributions and their approximations}

\author{H. L. Gan\footnote{Department of Mathematics and Statistics, the University of Melbourne, 
Parkville, VIC 3010, Australia. E-mail: ganhl@ms.unimelb.edu.au}\\ University of Melbourne    
\and    
A. Xia\footnote{Department of Mathematics and Statistics, the University of Melbourne, 
Parkville, VIC 3010, Australia. E-mail: aihuaxia@unimelb.edu.au 
\newline %\newline
Work supported in part by Australian Research Council Grants No DP120102398}.
\\    
University of Melbourne}    
\date{5 February 2014}
\maketitle

\begin{abstract}
Given a sequence of $n$ identically distributed random variables with common distribution $F$, the \emph{fragility distribution of order $m$}, represented by $\FD$, is the limit conditional distribution of the number of exceedances given there are at least $m$ exceedances, as the threshold tends to the right end point of $F$. In this paper we are concerned with the existence of $\FD$ and its asymptotic behaviour when $n$ becomes large. For a stationary sequence with its exceedance process converging to a compound Poisson process, we {derive} an explicit formula for calculating $\lim_{n \to \infty} \FD$. We also establish Stein's method for estimating the errors involved in fragility distribution approximations. %Examples are provided to illustrate what we will gain in using Stein's method.
\end{abstract}



\section{Introduction}
The number of earthquake related claims made to an insurance company is typically zero. However, given the event that at least one claim was made, it is highly likely that multiple claims were lodged. A question worth considering then is, given at least one earthquake related claim occurred, what is the distribution of the total number of claims? This idea extends to the more general question: if there are at least $m$ extreme events, what is the distribution of the number of extreme events? Aspects of this idea have been formalised in terms of the fragility index of order $m$, first introduced by Geluk et al.~(2007). Given a stationary sequence of random variables $X_1, \ldots , X_n$ with common distribution function $F_X$, define the number of exceedances above the threshold $s$ as:
\[ N_{s,n}: = \sum_{i=1}^n \mathbf{1}_{(s, \infty )}(X_i). \]
The extended fragility index of order $m$, denoted by $FI_n(m)$, is the asymptotic expected number of exceedances given that there are at least $m$ exceedances:
\[ FI_n(m) := \lim_{s \nearrow} \E(N_{s,n} | N_{s,n}\geq m), \]
where $s \nearrow$ is interpreted as ``when $s$ approaches the right end point $x_F:=\sup\{t:\ F_X(t)<1\}$ of $F$ from below". 

It is well-known in statistics that expectations do not carry sufficient information for statistical inferences. In this paper, we will instead consider the fragility distribution of order $m$ defined as%, represented by $\FD$,  
\[\FD:= \lim_{s \nearrow}\mathcal{L}\left(N_{s,n} | N_{s,n} \geq m\right).\]
In the case where $X_1,\ldots,X_n$ are independent and identically distributed (i.i.d.) with $\lim_{s\nearrow}(1-F_X(s))=0$, it is simple to show $\FD(\{m\})=1$ for all $m\leq n$. {In this case $\FD$ exists for all $m\le n$. In general though, given that $\FD$ exists for some natural numbers $m$, does this imply the existence of $\FD$ for other natural numbers? This question will be addressed in section~2.}

Ultimately, it is the dependence structure that characterises the properties of $\FD$. In section~3, we explore $\FD$ for a stationary sequence as $n$ tends towards infinity. Hsing et\ al.~(1998) showed that under a mixing condition, if the exceedance point processes converge to a limit distribution, then the limit is of a compound Poisson type. Using this result, we establish a relation between the limiting compound Poisson distribution and fragility distributions, and give an explicit formula for calculating $\lim_{n \rightarrow \infty}\FD$ for all $m\ge1$. 

In applications, we often face a fixed $n$ and hence it is of interest to know the errors involved when approximations are used to replace the actual fragility distribution. In section 4 we focus on estimating errors of a conditional compound Poisson approximation using Stein's method. However, similar to compound Poisson approximation, Stein's constants for conditional compound Poisson approximation are generally crude and are of little value unless specific conditions are satisfied. This leads us to investigate Stein's factors for conditional compound Poisson approximation when the compounding distribution satisfies a {certain} condition and conditional negative binomial approximation. Finally, examples are provided to show that the errors are typically small in applications when these approximations are used to replace the fragility distributions. 

%%%%%
\section{Existence of fragility distributions}

%In this section, we fix $n$, and for convenience, we write $N_s:=N_{s,n}$.
We consider the relationship between fragility distributions of different orders. By formulating the fragility distribution in the following manner the relationship between $\FD$ and $FD_{n,m+1}$ becomes clearer. For all $A \subset\{m, m+1, \ldots , n\} $,
\begin{align}
\FD(A) = \lim_{s \nearrow} \frac{\Pr(N_{s,n}\in A)}{\Pr(N_{s,n} \geq m)} %\notag \\
	% &= \lim_{s \nearrow} \frac{\Pr(N_{s,n}\in A)}{\Pr(N_{s,n} \geq m+1) + \Pr(N_{s,n}=m)} \notag \\
	 = \lim_{s \nearrow} \frac{\frac{\Pr(N_{s,n}\in A)}{\Pr(N_{s,n} \geq m+1)}}
							  {1 + \frac{\Pr(N_{s,n}=m)}{\Pr(N_{s,n} \geq m+1)}}. \label{bigeq}
\end{align}
Notice that the numerator in (\ref{bigeq}) yields $FD_{n,m+1}$ if $\lim_{s \nearrow} \frac{\Pr(N_{s,n}\in A)}{\Pr(N_{s,n} \geq m+1)}$ exists. From this formulation we can see that if for all $A$, two of
$\FD(A)$,
$FD_{n,m+1}(A)$ and
$\lim_{s \nearrow} \frac{\Pr(N_{s,n} =m )}{\Pr(N_{s,n} \geq m+1)}$ exist,
then the existence of the third is ensured. Hence, whether the existence of $\FD$ implies the existence of $FD_{n,m+1}$ or vice versa depends on the existence of $\lim_{s\nearrow} \frac{\Pr(N_{s,n} =m)}{\Pr(N_{s,n} \geq m+1)}$.

The following counterexample shows that the existence of $\FD$ does not guarantee the existence of $FD_{n,m+1}$. Neither does the existence of $FD_{n,m+1}$ guarantee the existence of $\FD$.

To start, we define a density function on $[0,1]$ as
$$
g_1(y) = \begin{cases}
 2 & y \in [1-\frac{1}{2^k}, 1-\frac{3}{2^{k+2}}] \text{ for }k \in \Z_+ := \{0,1,2,\ldots\}{,} \\
 0 & \text{otherwise}{.}
\end{cases}
$$
Let $G_1(y)$ denote the distribution function of $g_1$. %An illustration of the distribution function can be found in the appendix. 
By considering the two sequences $y_k = 1-\frac{1}{2^k}$ and $y_k = 1-\frac{3}{2^{k+1}}$, it can be shown that
\begin{equation}\lim_{y \rightarrow 1}\frac{1-G_1(y)}{1-y}\label{xiaadd05}\end{equation}
does not exist.

Now we present an example where $FD_{n,2}$ exists, but $FD_{n,1}$ does not.

Consider a two dimensional random vector on the unit square, where the density sits entirely upon {three} lines
$L_1 = \{(x_1,0),\text{ } 0 < x_1 \leq 1\}$,
$L_2 = \{(0,x_2), \text{ }0 \le x_2 \leq 1\}$,
 $L_{12} = \{(x_1,x_2),\text{ } 0< x_1=x_2\leq1\}$.
We use the (one-dimensional) Lebesgue measure on each of the three line segments and put rescaled uniform densities on $L_1$ and $L_2$, and a rescaled density of $g_1$ on the diagonal. Hence our density is

\begin{equation*}
h_1(x_1,x_2) = \begin{cases}
\frac{1}{c_1} & (x_1,x_2)\in L_1 \cup L_2,\\
\frac{1}{c_1}g_1\left(\frac{r}{\sqrt2}\right) & (x_1,x_2) \in L_{12},\\
0 & \text{otherwise},
\end{cases}
\end{equation*}
where
$r =\sqrt{x_1^2 + x_2^2}, \text{ } c_1 = 2 + \sqrt{2}.$
It is easy to see that $FD_{2,2}(\{2\}) = 1$, as the number of exceedances is at least 2, but can not exceed 2. However,
\begin{align*}
\frac{\Pr(N_{s,2} = 1)}{\Pr(N_{s,2} \geq 2)} = \frac{2\Pr(X_1> s, X_2 = 0)}{\Pr(X_1 = X_2 \geq s)}
								    %= \frac{2\frac{1-s}{c_1}}{\frac{\sqrt{2}}{c_1}(1-G_1(s))} 
								    = \sqrt{2} \cdot \frac{1-s}{1-G_1(s)},
\end{align*}
which, according to \Ref{xiaadd05}, does not converge as $s\nearrow$. Therefore, even though $FD_{2,2}$ exists, $FD_{2,1}$ does not.

We now construct another counter-example to show that despite the existence of $FD_{3,1}$, $FD_{3,2}$ does not exist. To this end,
let the joint density of $(X_1, X_2, X_3)$ lie only on the following lines:
$L_1 = \{(x_1,0,0),\text{ } 0 < x_1 \leq 1\},$
$L_2 = \{(0,x_2,0), \text{ }0 < x_2 \leq 1\},$
$L_3 = \{(0,0,x_3), \text{ }0 < x_3 \leq 1\},$
$L_{12} = \{(x_1,x_2,0),\text{ } 0< x_1=x_2\leq1\},$
$L_{23} = \{(0,x_2,x_3),\text{ } 0< x_2=x_3\leq1\},$
$L_{13} = \{(x_1,0,x_3),\text{ } 0< x_1=x_3\leq1\},$
$L_{123} = \{(x_1,x_2,x_3),\text{ } 0\leq x_1=x_2=x_3\leq1\},$
equipped with the (one-dimensional) Lebesgue measure on these lines.
We define a new distribution function
\begin{equation*}
G_2(z) = \begin{cases}
0 & z < 0,\\
1-(1-z)(1-G_1(z)) & z \in [0,1],\\
1 & \text{otherwise},
\end{cases}
\end{equation*}
and denote its density by $g_2(z)$.
We then set up our joint density of $(X_1,X_2,X_3)$ as 
%with rescaled uniforms on the axes as before, a linear density on the two-dimensional diagonals, and a density defined by {$g_2$} on the three-dimensional diagonal:
\begin{equation*}
h_2(x_1,x_2,x_3) = \begin{cases}
\frac{1}{c_2} & (x_1,x_2,x_3)\in L_1 \cup L_2 \cup L_3,\\
\frac{1}{c_2}\left(\sqrt{2}-r\right) & (x_1,x_2,x_3) \in L_{12} \cup L_{23} \cup L_{13},\\
\frac{1}{c_2} g_2\left(\frac{r}{\sqrt{3}}\right) & (x_1,x_2,x_3) \in L_{123}, \\
0 & \text{otherwise},
\end{cases}
\end{equation*}
where
$ r= \sqrt{x_1^2 + x_2^2 + x_3^2} \text{, } c_2 = 6 + \sqrt{3}.$
Now, 
\begin{eqnarray*}
\Pr(N_{s,3}=1|N_{s,3}\geq1) =\frac{3}{\sqrt{3}(1-G_1(s)) + 3(1-s) + 3}
				       \rightarrow 1\mbox{ as }s\nearrow,
\end{eqnarray*}
so $FD_{3,1}$ exists. On the other hand, using \Ref{xiaadd05}, one can show
\begin{eqnarray*}
\Pr(N_{s,3}=2 | N_{s,3}\geq2) 
			%&=& \frac{3(1-s)^2}{\sqrt{3}(1-s)(1-G_1(s)) + 3(1-s)^2}\\
		       = \frac{\sqrt{3}}{ \frac{1-G_1(s)}{1-s} + \sqrt{3}},
\end{eqnarray*}
does not converge as $s \nearrow$  and hence $FD_{3,2}$ does not exist.

\section{Stationary Sequences}

In the previous section we considered the fragility distribution in the case when the number of random variables {was} fixed and finite. We now study properties of $\lim_{n \rightarrow \infty} \FD$ for stationary sequences 
$\{X_i,\ i\ge 1\}$ with distribution $F_X$ satisfying 
 \begin{equation}
 \lim_{s\nearrow}\frac{1-F_X(s)}{1-F_X(s^-)}=1.\label{xiaadd02}
 \end{equation} 
Let $\fn_{u_n,n}(B)=\sum_{i=1}^n\ind_{\{\frac in\in B,X_i>u_n\}}$ for any Borel $B\subset [0,1]$, where $\{u_n\}$ is a sequence of constants approaching $x_F$. Hence $\fn_{u_n,n}$ is the time-scaled point process of exceedances and it serves as an instrument for using point process theory to obtain limiting properties in extreme value theory, see for example, Leadbetter~et~al.~(1983), chapter~5. 

We say a random variable $\cprv $ has a compound Poisson distribution $\CP(\bm\L)$ with $\bm\L=(\lambda_1,\lambda_2,\dots)$, if 
$\cprv  \stackrel{d}{=} \sum_{i=1}^\infty i X_i$, where $X_i$ follows Poisson distribution with mean $\lambda_i$, denoted by $\Po(\lambda_i)$, and the $X_i$'s are independent.
If we write $\L=\sum_{i=1}^\infty\L_i$ and define $\pi_i=\frac{\L_i}{\L}$, $i\ge 1$, then $\cprv $ can also be represented as the sum of a $\Po(\L)$ number of i.i.d. random variables with cluster distribution $\bp:=\{\pi_i\}$.
It has been established since Hsing et al.~(1988) that under some mild conditions, the limiting distribution of 
$N_{u_n,n}=\fn_{u_n,n}([0,1])$ is necessarily compound Poisson. This observation relates the study of extreme value theory to the estimates of the accuracy of compound Poisson approximation for $\cL(N_{u_n,n})$, which can be found in, e.g., \cite{BCL,BNX,Er99,Er00,Michel87,N98,N03,NX12,Raab,Roos94}.

By taking appropriate percentiles of the underlying distribution, we can find a normalising sequence $\{\unt\}$ such that for any $0 < \tau < \infty$,
$$ n(1-F(u_n^{(\tau)})) \rightarrow \tau \hspace{1cm} \text{as }n \rightarrow \infty.$$
The existence of such sequences $\{\unt\}$ is guaranteed by the condition \Ref{xiaadd02} (see Leadbetter et\ al.~(1983), Theorem~1.7.13). In this section, assume that we are working with such sequences $\unt$ and condition \Ref{xiaadd02} holds. 
%We shall also be using the following mixing condition from Leadbetter et.\ al.~\cite{Leadbetter1983}.
%
%If $\{u_n\}$ is a sequence of constants, for any $1 \leq i \leq j \leq n$, let $\mathscr{B}_i^j(u_n)$ be the $\sigma$-algebra generated by the events $\{X_\kappa \leq u_n\}$ for any $i \leq \kappa \leq j$. For each $n$ and $1 \leq l \leq n-1$, {define}
%\begin{align*}
%\alpha_{n,l} = \sup \{|\Pr(A \cap B) - \Pr(A)\Pr(B)|: 
%		   A \in \mathscr{B}_1^k(u_n), B \in \mathscr{B}_{k+l+1}^n(u_n), 1 \leq k \leq n-l-1\}{.}
%\end{align*}
%$\{X_n\}$ is said to satisfy the condition $\Delta(u_n)$ if $\alpha_{n,l_n} \rightarrow 0$ as $n \rightarrow \infty$ for some sequence $\{l_n\}$ with $l_n = o(n)$.

\begin{theorem}\label{limit}
Suppose that $\FD$ exists for sufficiently large $n$ and for any $0<\tau \le \tau_0$ and $k \geq m$, $\Pr(N_{\unt,n} =k| N_{\unt,n} \geq m)$ converges uniformly in $n$ to $\Pr(N^\tau=k|N^\tau \geq m)$, where
$N^\tau$ is a compound Poisson random variable with rate $\theta \tau$, $\theta>0$, and cluster distribution $\bp$. Then for any $A \subset \Z_m:=\{ m, m+1, \ldots \}$,
\begin{align*}
\lim_{n \rightarrow \infty} \FD(A) 
	&=\frac{\bp^{*\cI_m}(A)}{\bp^{*\cI_m}(\Z_m)},
\end{align*}
where $\bp^{*j}$ is the convolution of $\bp$ $j$ times and
$\cI_m = \min\{i : \bp^{*i}(\Z_m) > 0\}$. In particular, we have
$\lim_{n \rightarrow \infty} \FDo=\bp.$

\end{theorem}
\begin{proof} Applying Theorem~7.11 in Rudin~(1976), one can see that the uniform convergence allows the exchange of limits, {giving}
\begin{eqnarray}
\lim_{n \rightarrow \infty}\FD(A) &=& \lim_{n \rightarrow \infty} \lim_{s \nearrow} \Pr(N_{s,n} \in A | N_{s,n} \geq m)\nonumber\\
&=&\lim_{n \rightarrow \infty} \lim_{\tau \rightarrow 0} \Pr(N_{\unt,n} \in A | N_{\unt,n} \geq m)\nonumber\\
&=& \lim_{\tau \rightarrow 0}\lim_{n\to\infty} \Pr(N_{\unt,n}\in A | N_{\unt,n}\geq m). \label{xiaadd09}
\end{eqnarray}
Since $N^\tau$ follows compound Poisson distribution with rate $\theta \tau$ and compounding distribution $\bp$, we can write 
\[N^\tau\stackrel{d}{=}\sum_{i=1}^{P_\tau} \xi_i ,\]
where
$P_\tau$ is a Poisson random variable with mean $\theta \tau$, $\xi_1,\xi_2,\ldots$ are i.i.d. with distribution $\bp$ and independent of $P_\tau$.
Using the law of total probability, by conditioning on $P_\tau$, we obtain from \Ref{xiaadd09} that
for $A\subset \Z_m$,
\begin{align*}
\lim_{n \rightarrow \infty} \FD(A) &= \lim_{\tau \rightarrow 0} \Pr(N^\tau \in A |N^\tau\geq m)\\
	&=\lim_{\tau \rightarrow 0} \frac{\bp^{*\cI_m}(A) \frac{e^{-\theta \tau}(\theta \tau)^{\cI_m}}{\cI_m!} + o(\tau^{\cI_m})}{  \bp^{*\cI_m}(\Z_m) \frac{e^{-\theta \tau}(\theta \tau)^{\cI_m}}{\cI_m!}+ o(\tau^{\cI_m})}=\frac{\bp^{*\cI_m}(A)}{\bp^{*\cI_m}(\Z_m)}.
\end{align*}
Finally, for $m=1$, since $\bp(\Z_1)=1$, then $\cI_1=1$ and the conclusion follows.
\end{proof}

%\blue{Need an example to say that the uniform convergence is necessary}
%\rem{It remains an open problem whether the condition of uniform convergence can be lifted.}

In applications, we often face situations where $n$ is fixed and $u_n$ is chosen in such a way that the number of exceedances is within an acceptable range. That is, regardless how large the data set we have, we don't have sufficient information to obtain the exact fragility distribution. Hence, for practical purposes, our interest should be focused on suitable approximations of fragility distributions and their associated error estimates. On the other hand, although it has been established that $\cL(N_{s,n})$ can be reasonably approximated by a compound Poisson distribution (\cite{BCL,BNX,Michel87,N98,N03,NX12,Roos94}), the probability $\Pr(N_{s,n} \geq m)$ is usually small and the error estimates are often larger than the actual probabilities. For this reason, we can not rely on the estimates of the errors of $\Pr(N_{s,n} \in \cdot)$ and $\Pr(N_{s,n} \geq m)$ to work out the approximation errors of $\Pr(N_{s,n} \in \cdot | N_{s,n} \geq m)$ and it is necessary to study the estimates of $\Pr(N_{s,n} \in \cdot|N_{s,n} \geq m)$ directly.

%%%%%%%%%%%
% Approximation section
%%%%%%%%%%%
\section{Conditional compound Poisson Approximation}
In the previous section we have seen that the limit fragility distribution can be evaluated by the conditional compound Poisson limit. In this section we will focus on estimating errors of conditional compound Poisson approximation via Stein's method.

%\subsection{Conditional compound Poisson approximation via birth-death processes}
 For any random variable $X$, we write $X^{\resm}$ as a random variable having the distribution $\mathcal{L}(X| X \geq m)$, where $m$ is a non-negative integer. For convenience, we define $\CP^{\resm}(\bm{\L}):=\mathcal{L}(\cprv ^{\resm})$, where $\cprv  \sim \CP(\bm{\lambda})$. The following lemma can be directly verified.

\begin{lem}
For a non-negative integer $m$, $W \sim \CP^{\resm}(\bm{\L})$ if and only if for all bounded functions $g_m$ on $\Z_m$,
\begin{align}
\E \left[ \sum_{j=1}^\infty j \L_jg_m(W+ j) - W g_m(W) \ind_{W > m} \right] = 0.\label{steinidentityCP}
\end{align}
\end{lem}
Let ${\cal B}_m g_m(i):=\sum_{j=1}^\infty j \L_jg_m(i+ j) - i g_m(i) \ind_{i > m}$. Our interest is to assess the difference between two distributions $Q_1$ and $Q_2$ on $\Z_m$ so we define the total variation distance as
$$d_{TV}(Q_1,Q_2):=\sup_{f\in\mathcal{F}_m}\left|\int fdQ_1-\int fdQ_2\right|,$$
where  $\mathcal{F}_m:=\{{\bf 1}_A:\ A\subset \Z_m\}$.
We write Stein's equation as
\begin{equation}
{\cal B}_m g_m(i)=f(i)-\CP^{\resm}(\bm\L)\{f\}, \ f\in\mathcal{F}_m\label{SteineqCP}
\end{equation}
where $\CP^{\resm}(\bm\L)\{f\} := \E f(\cprv ^{\resm})$ with $\cprv ^{\resm} \sim \CP^{\resm}(\bm\L)$. Using the same argument as in Theorem~1 in Barbour, Chen \& Loh~(1992), one can prove that the equation \Ref{SteineqCP} has a solution $g_{m,f}$ defined on $\Z_m$ and the solution is unique except at $i=m$. 

For a function $h$ on $\Z_m$, we write $\Delta h(\cdot)=h(\cdot+1)-h(\cdot)$ and $\|h\|_m=\sup_{w\in \Z_m}|h(w+1)|.$
To apply Stein's method, bounds for 
\begin{align}
G_{m,1}= \sup_{f \in \mathcal{F}_m}\|g_{m,f}\|_m,\ \ \ 
G_{m,2} = \sup_{f \in \mathcal{F}_m}\| \Delta g_{m,f} \|_m\label{G1G2}
\end{align}
are needed. 
However, as was demonstrated in Barbour, Chen \& Loh~(1992) and Barbour \& Utev~(1998, 1999), the estimates of Stein's factors for general compound Poisson approximation are unsatisfactory and useful estimates are only available for special cases. Consequently, we will deal with two special cases: (1) $i\lambda_i$ is decreasing in $i$ and (2) conditional negative binomial approximation.

\subsection{Case 1: $i\lambda_i$ is monotone decreasing in $i$}\label{sec:case1}

\begin{theorem}\label{thmCPdecreasing} If $i\L_i$ is a decreasing function of $i$, then both $G_{m,1}$ and $G_{m,2}$ are decreasing in $m$.
\end{theorem}

\begin{proof} By setting $g_{m,f}(i) = \Delta h_{m,f}(i-1)$ (see Barbour (1998) and Barbour, Chen \& Loh~(1992)), if we assume $i\L_i$ is decreasing in $i$, then the form of Stein's identity in \Ref{steinidentityCP} naturally leads to a generator interpretation %of a  birth-death process $\{Z^{\resm}(t)\}$ on $\mathbb{Z}_m$ 
with generator
\begin{align*}
\mathcal{A}_mh_{m,f}(i)%&=\sum_{j=1}^\infty j \L_j \left( h_{m,f}(i+j) - h_{m,f}(i+j-1)\right) + i\left(h_{m,f}(i-1) - h_{m,f}(i) \right) \ind_{i >m}\notag\\
	&:= \sum_{j=1}^\infty \left[\left( j \L_j - (j+1) \L_{j+1}\right) (h_{m,f}(i+j) - h_{m,f}(i))\right] \notag\\
	&\ \ \ \ +i(h_{m,f}(i-1) - h_{m,f}(i)) \ind_{i > m},\ i\in\Z_m.
\end{align*}
%with birth rate of $j \L_j - (j+1) \L_{j+1}$ for $j$ births, and a unit per capita death rate for $i \geq m+1$. 
%\begin{align} \mathcal{A}h_{m,f}(i) = f(i) - \CP^{\resm}(\bm\L)\{f\},\label{steineqCP}\end{align}
Furthermore, it can be verified that the solution for $h_{m,f}(i)$ to Stein equation 
\begin{align*} \mathcal{A}_mh_{m,f}(i)=\mathcal{B}_mg_{m,f}(i)= f(i) - \CP^{\resm}(\bm\L)\{f\},\end{align*}
is
\begin{align*}
h_{m,f}(i) = -\int_0^\infty \left( \E f(Z_i^{\resm}(t)) - \E f(\cprv ^{\resm}) \right) dt,
\end{align*}
where $Z_i^{\resm}$ is a birth-death process with generator $\mathcal{A}$ and initial value $Z_i^{\resm}(0)=i$.
Using the strong Markov property of the birth-death process, we obtain
\begin{align}
h_{m,f}(i+1) &= -\int_0^\infty \left( \E f(Z_{i+1}^{\resm}(t)) - \E f(\cprv ^{\resm})\right)dt\notag\\
	&= -\E\int_0^{\tau_{i+1,i}^{\resm}} \left( f(Z_{i+1}^{\resm}(t)) - \E f(\cprv ^{\resm})\right)dt + h_{m,f}(i), \label{gdecomp}
\end{align}
where $\tau_{i+1,i}^{\resm} = \inf\{t : Z_{i+1}^{\resm}(t) = i\}$. One can replace $f \in \cF_m$ with $\ind_{\Z_m} - f$ to show that 
$$G_{m,1} = \sup_{f \in \cF_m} \sup_{i \in \Z_m} g_{m,f}(i+1)=-\inf_{f \in \cF_m} \inf_{i \in \Z_m} g_{m,f}(i+1),$$
hence, we can assume that $h_{m,f}(i+1) - h_{m,f}(i) \leq 0$. Now,
\begin{align}
h_{m,f}(i+1) - h_{m,f}(i) &= -\E\int_0^{\tau_{i+1,i}^{\resm}} \left( f(Z_{i+1}^{\resm}(t)) - \E f(\cprv ^{\resm})\right)dt\notag\\
	&=  -\E\int_0^{\tau_{i+1,i}^{\resm}} \left(  f(Z_{i+1}^{\resm}(t)) - \E f(\cprv ^{\resmone})\right)dt\notag\\
	&\phantom{VV} + \left[ \E f(\cprv ^{\resm}) - \E f(\cprv ^{\resmone}) \right] \E \tau_{i+1,i}^{\resm}\notag\\
	&= h_{m-1,f}(i+1) - h_{m-1,f}(i)\notag\\
	&\phantom{VV} + \left[ \E f(\cprv ^{\resm}) - \E f(\cprv ^{\resmone}) \right] \E \tau_{i+1,i}^{\resm},\label{conditionalCP01}
\end{align}
where the last equality is because $(Z_{i+1}^{\resm}(\cdot)\ind_{\cdot<\tau_{i+1,i}^{\resm}},\tau_{i+1,i}^{\resm})\stackrel{d}{=}(Z_{i+1}^{\resmone}(\cdot)\ind_{\cdot<\tau_{i+1,i}^{\resmone}},\tau_{i+1,i}^{\resmone})$ for $i\ge m$.
Noting that $f\in\mathcal{F}_m$, we have
\begin{align}
\E f(\cprv ^{\resm}) - \E f(\cprv ^{\resmone}) %&= \frac{\E f(P) \ind_{P \geq m}}{\Pr(P \geq m)} - \frac{\E f(P) \ind_{P \geq m-1}}{\Pr(P \geq m-1)}\\
	=\frac{\E f(\cprv ) }{\Pr(\cprv  \geq m)} - \frac{\E f(\cprv ) }{\Pr(\cprv  \geq m-1)}
	\geq 0.\label{conditionalCP02}
\end{align}
This, together with \Ref{conditionalCP01}, implies $h_{m-1,f}(i+1) - h_{m-1,f}(i) \leq h_{m,f}(i+1) - h_{m,f}(i)$. Therefore,
\begin{align*}
-G_{m,1} &= \inf_{f \in \cF_m} \inf_{i \in \Z_m} \left(h_{m,f}(i+1) - h_{m,f}(i)\right)\\
	&\geq  \inf_{f \in \cF_m} \inf_{i \in \Z_m} \left(h_{m-1,f}(i+1) - h_{m-1,f}(i)\right)\\
	&\geq \inf_{f \in \cF_{m-1}} \inf_{i \in \Z_{m-1}} \left(h_{m-1,f}(i+1) - h_{m-1,f}(i) \right)= -G_{m-1,1}.
\end{align*}

The proof of the monotonicity of $G_{m,2}$ is similar. Again replacing $f \in \cF_m$ with $1_{\Z_m}-f$ if necessary, we can prove $G_{m,2} = \sup_{f \in \cF_m} \sup_{i \in \Z_m} \Delta g_{m,f}(i+1)=-\inf_{f \in \cF_m} \inf_{i \in \Z_m} \Delta g_{m,f}(i+1)$. Hence we can assume $\Delta_2h_{m,f}(i):= h_{m,f}(i+2) - 2h_{m,f}(i+1) + h_{m,f}(i) \geq 0$.
Arguing in the same way as for \Ref{conditionalCP01}, we have
\begin{align*}
\Delta_2h_{m,f}(i) &= - \E \int_0^{\tau_{i+2,i+1}^{\resm}} (f(Z_{i+2}^{\resm}(t)) - \E f(\cprv ^{\resm}))dt\\
	&\phantom{V} + \E \int_0^{\tau_{i+1,i}^{\resm}} (f(Z_{i+1}^{\resm}(t)) - \E f(\cprv ^{\resm}))dt\\
	&= - \E \int_0^{\tau_{i+2,i+1}^{\resm}} (f(Z_{i+2}^{\resm}(t)) - \E f(\cprv ^{\resmone}))dt\\
	&\phantom{V} + \E \int_0^{\tau_{i+1,i}^{\resm}} (f(Z_{i+1}^{\resm}(t)) - \E f(\cprv ^{\resmone}))dt\\
	&\phantom{V} +(\E f(\cprv ^{\resm}) - \E f(\cprv ^{\resmone}))(\E \tau_{i+2,i+1}^{\resm} - \E \tau_{i+1,i}^{\resm})\\
	&= \Delta_2h_{m-1,f}(i) + (\E f(\cprv ^{\resm}) - \E f(\cprv ^{\resmone}))(\E \tau_{i+2,i+1}^{\resm} - \E \tau_{i+1,i}^{\resm}).
\end{align*}
It can be shown via a coupling that $\E \tau_{i+2,i+1}^{\resm} - \E \tau_{i+1,i}^{\resm} \leq 0$. Hence, it follows from \Ref{conditionalCP02} that $\Delta_2h_{m,f}(i) \leq \Delta_2h_{m-1,f}(i)$, which ensures
\begin{align*}
G_{m,2} &= \sup_{f \in \cF_m} \sup_{i \in \Z_m} \Delta_2h_{m,f}(i+1)\leq  \sup_{f \in \cF_m} \sup_{i \in \Z_m} \Delta_2h_{m-1,f}(i+1)\\
	&\leq \sup_{f \in \cF_{m-1}} \sup_{i \in \Z_{m-1}} \Delta_2h_{{m-1},f}(i+1) = G_{m-1,2}.
\end{align*}
\end{proof}
As a direct consequence of the above theorem, we can use the bounds for compound Poisson approximation given in Barbour, Chen \& Loh (1992) to give crude estimates of $G_{m,1}$ and $G_{m,2}$.
\begin{corollary}
If $j \L_j$ is a decreasing function of $j$, then for any $m$,
\begin{align*}
G_{m,1} &\leq \begin{cases} 1 & \text{if } \L_1 - 2\L_2 \leq 1%\label{28jan14-1}
\\
(1/\sqrt{\L_1 - 2\L_2})\left[ 2 - (1/\sqrt{\L_1-2\L_2}) \right] & \text{if }\L_1 - 2\L_2 > 1\end{cases},\\
G_{m,2} &\leq 1 \wedge \frac{1}{\L_1 - 2\L_2} \left[ \frac{1}{4(\L_1 - 2\L_2)} + \log^+2(\L_1 - 2\L_2) \right].%\label{28jan14-2}
\end{align*}
\end{corollary}
We will see in the next subsection that the bounds can be improved for conditional Poisson and conditional negative binomial approximations. However, in the general case, it remains a challenging open problem to find the optimal estimates of $G_{m,1}$ and $G_{m,2}$. %but given the great generality of the compound Poisson distribution, it is unlikely that there is room for much improvement. 
% As a result, we instead investigate a more specific case, namely negative binomial approximation, of which the Poisson distribution is a special limiting case.

\subsection{Case 2: Conditional Negative Binomial Approximation}
We use the following parameterisation of the negative binomial distribution. Let a random variable $Z$ have negative binomial distribution with parameters $r$ and $p$, denoted by $\Nb(r,p)$, if it has mass function
\[ \Pr(Z = k) = \frac{\Gamma(r+k)}{\Gamma(r)k!} (1-p)^r p^k, \text{ }k = 0,1,\ldots; r>0, 0<p<1.\]
The negative binomial distribution can be viewed as a compound Poisson distribution with cluster distribution following 
a logarithmic distribution (see Johnson et al.~(2005), p.~223). It can also be considered as the stationary distribution of a birth-death process with linear birth rate and unit per capita death rate (Phillips (1996)). The latter consideration leads us to the following observation.
 
\begin{lem}
For a non-negative integer $m$, $W \sim \NB^{\resm}(r,p)$ if and only if for all bounded functions $g_m$ on $\Z_m$,
\begin{align}
\E\left[p(r+W)g_m(W+1) - Wg_m(W)\ind_{W>m} \right] = 0. \label{steinidentity}
\end{align}
\end{lem}

Again we solve for the function $g_m := g_{m,f}$ that satisfies Stein's equation
\begin{align} \B_mg_m(i):=p(r+i)g_m(i+1) - ig_m(i)\ind_{i>m} = f(i) - \NB^{\resm}(r,p)\{f\},\label{Steineq}\end{align}
where $\NB^{\resm}(r,p)\{f\} := \E(f(Z^{\resm}))$ with $Z^{\resm} \sim \NB^{\resm}(r,p)$, and calculate bounds for $G_{m,1}$ and $G_{m,2}$.

\begin{theorem}\label{conditonalNBth1}
For conditional negative binomial approximation, both $G_{m,1}$ and $G_{m,2}$ are decreasing in $m$.
\end{theorem}
We omit the proof of this theorem as it is essentially the same as the proof of Theorem~\ref{thmCPdecreasing}.

Using the bound for $G_{0,1}$ in Brown and Phillips (1999), and the previous theorem we achieve the following. 
\begin{corollary}
For conditional negative binomial approximation, the solution to Stein's equation \Ref{Steineq} satisfies
\begin{align*} 
G_{m,1} \leq \frac{1}{1-p} \wedge \frac{1.75}{\sqrt{rp(1-p)}}.
\end{align*}
\end{corollary}

Unlike in {the compound Poisson case} where we use the unconditional bound of $G_{0,2}$ to bound $G_{m,2}$, we can obtain a tight bound of $G_{m,2}$ {for negative binomial approximation}.

\begin{theorem}\label{thmdeltag} For conditional negative binomial approximation, the solution to Stein's equation \Ref{Steineq} satisfies
\begin{align} G_{m,2} = \frac{\Pr(Z > m)}{p(r+m)\Pr(Z \geq m)},\label{deltag}\end{align}
where $Z \sim \NB(r,p)$.
\end{theorem}
\begin{proof}
Similarly to the previous subsection, by setting $g_m(i) = h_m(i) - h_m(i-1)$, the form of Stein's identity in \Ref{steinidentity} leads to a generator interpretation with generator
\begin{align*}\A_m h_m(i):=p(r+i)(h_m(i+1)-h_m(i)) + i(h_m(i-1)-h_m(i))\ind_{i>m}.\end{align*}
%with birth rates $\alpha^{\resm}_i = p(r+i)$ and per-capita death rates as long as $Z^{\resm} > m$.
It is a routine exercise to show that the stationary distribution of a process with this generator is $\NB^{\resm}(r,p)$.
Noting that our process satisfies (C4) in Brown \& Xia~(2001), we obtain from Theorem~2.10 of Brown \& Xia~(2001) that
\begin{align*}
\sup_{f \in \mathcal{F}_m} |\Delta g_{m,f}(i)| %&= \frac{\Pr(Z^{\resm} \geq i+1)}{\alpha^{\resm}_i} + \frac{\Pr(Z^{\resm} \leq i-1)}{\beta^{\resm}_i}\\
	&= \frac{1 - \pi_m^{\resm} - \ldots - \pi_i^{\resm}}{p(r+i)} + \frac{\pi_m^{\resm} + \ldots + \pi_{i-1}^{\resm}}{i},
\end{align*}
where $\pi_j^{\resm} = \Pr(Z^{\resm} = j)$. Using the balance equations of stationary processes, we also have that $(j+1)\pi_{j+1}^{\resm} = p(r+j)\pi_j^{\resm}$. Therefore, rearranging the above equation gives
\begin{align}
\sup_{f \in \mathcal{F}_m} |\Delta g_{m,f}(i)| %&= \frac{1-\pi_m^{\resm}}{p(r+i)} + \sum_{k=m}^{i-1} \left(\frac{\pi_k^{\resm}}{i} - \frac{\pi_{k+1}^{\resm}}{p(r+i)}\right)\\
	%&= \frac{1-\pi_m^{\resm}}{p(r+m)} + (1-\pi_m^{\resm})\left(\frac{1}{p(r+i)} - \frac{1}{p(r+m)}\right) + \sum_{k=m}^{i-1} \left(\frac{\pi_k^{\resm}}{i} - \frac{\pi_{k+1}^{\resm}}{p(r+i)}\right)\\
	&= \frac{1-\pi_m^{\resm}}{p(r+m)} + (\pi_{i+1}^{\resm} + \pi_{i+2}^{\resm} + \ldots)\frac{m-i}{p(r+i)(r+m)}\notag\\
	&\phantom{V}+\sum_{k=m}^{i-1} \left(\frac{\pi_k^{\resm}}{i} - \frac{\pi_{k+1}^{\resm}}{p(r+m)}\right)\notag\\
	%&\leq \frac{1-\pi_m^{\resm}}{p(r+m)} + \sum_{k=m}^{i-1} \left(\frac{\pi_k^{\resm}}{i} - \frac{\pi_{k+1}^{\resm}}{p(r+m)}\right)\\
	%&= \frac{1-\pi_m^{\resm}}{p(r+m)} + \sum_{k=m}^{i-1} \pi_k^{\resm} \left( \frac{1}{i} - \frac{r+k}{(k+1)(r+m)} \right)\\
	%&\leq \frac{1-\pi_m^{\resm}}{p(r+m)} + \sum_{k=m}^{i-1} \pi_k^{\resm} \left( \frac{1}{k+1} - \frac{r+k}{(k+1)(r+m)} \right)\\
	%&= \frac{1-\pi_m^{\resm}}{p(r+m)} + \sum_{k=m}^{i-1} \pi_k^{\resm} \left(\frac{m-k}{(k+1)(r+m)} \right)\\
	&\leq \frac{1-\pi_m^{\resm}}{p(r+m)},\label{CNB01}
\end{align}
since $m\le i$ and for $m\le k\le i-1$,
$$\frac{\pi_k^{\resm}}{i} - \frac{\pi_{k+1}^{\resm}}{p(r+m)}= \pi_k^{\resm} \left( \frac{1}{i} - \frac{r+k}{(k+1)(r+m)}\right)\le \frac{\pi_k^{\resm}(m-k)}{(k+1)(r+m)}\le 0.
$$
Direct verification ensures that the inequality of \Ref{CNB01} becomes an equality when $i=m$.
\end{proof}

For conditional Poisson approximation, it is natural to set Stein's equation
\begin{align} \C_mg_m(i):=\L g_m(i+1) - ig_m(i)\ind_{i>m} = f(i) - \Po^{\resm}(\L)\{f\},\label{steineqPo}\end{align}
where $\Po^{\resm}(\L)\{f\} := \E(f(P^{\resm}))$ with $P^{\resm} \sim \Po^{\resm}(\L)$. We define $G_{m,1}$ and $G_{m,2}$ as in \Ref{G1G2}. 

\begin{corollary}\label{poissonbounds}
For conditional Poisson approximation, the solution to Stein's equation \Ref{steineqPo} satisfies
\[G_{m,1} \leq 1\wedge \sqrt{\frac{2}{\lambda e}} , \]
\[G_{m,2} = \frac{\Pr(P > m)}{\lambda\Pr(P \geq m)},\]
where $P \sim \Po(\lambda)$.
\end{corollary}
\begin{proof} The bound of $G_{m,1}$ follows from Remark~3.4 of Barbour \& Brown~(1992) and Theorem~\ref{conditonalNBth1}. {For} 
the second estimate, since $\NB(r,p)$ approaches $\Pn(\L)$ when $r \rightarrow \infty$ and $rp \rightarrow \lambda$, therefore the bound follows by taking the limit in Theorem~\ref{thmdeltag}.
\end{proof}

\rem{The advantages of dealing with the conditional approximation directly can be seen as follows. For $\Po^\resone(\lambda)$ approximation, $G_{1,2} = \frac{1 - e^{-\lambda} - \lambda e^{-\lambda}}{\lambda (1-e^{-\lambda})}.$ In the case where $\lambda$ is small, this tends towards $\frac{1}{2}$. If one were to use Stein's method for Poisson approximation, then $G_{0,2} = \frac{1 - e^{-\lambda} }{\lambda}$, which has a limit of 1 for small $\lambda$. Therefore, using the conditional approximation appropriately, one can typically reduce the error bound by a factor of $\frac{1}{2}$ when $\lambda$ is small.}

It is worthwhile to point out that the representation \Ref{steinidentity} enables us to relate an estimate of conditional negative binomial approximation to that of the corresponding negative binomial approximation. This property does not seem to be shared by the general conditional compound Poisson approximation.

\begin{lem}\label{uncond-cond-thm}
If a nonnegative integer valued random variable $W$ can be approximated by $\NB(r,p)$ and it can be shown that
$$
  |\E\B_0 g_0(W)| \leq \epsilon_1 \|g_0\|_0 + \epsilon_2 \|\Delta g_0\|_0,
$$
for all functions~$g_0$ on $\Z_0$ for which $\|g_0\|_0$ and $\|\Delta g_0\|_0$ are finite, 
%\begin{align}d_{TV}(\mathcal{L}(W), \CP(\bm\L)) \leq \epsilon_1 G_{0,1} + \epsilon_2 G_{0,2},\label{unconditional}\end{align} 
and $\epsilon_1$, $\epsilon_2$ are positive, then $W^{\resm}$ can be approximated by $\Nb^{\resm}(r,p)$ with
$$d_{TV}(\mathcal{L}(W^{\resm}), \Nb^{\resm}(r,p)) \leq \frac{1}{\Pr(W \geq m)} \left\{\epsilon_1 G_{m,1} + \epsilon_2 G_{m,2}\right\}.
$$
\end{lem}
\begin{proof} For each $f\in\mathcal{F}_m$, we define $L_f(w)=\left\{\begin{array}{ll}g_{m,f}(w),&\mbox{for }w>m,\\
0,&\mbox{for }w\le m,\end{array}\right.$ then it follows from \Ref{Steineq} that
\begin{eqnarray*}
&&\left|\E f(W^{\resm})-\Nb^{\resm}(r,p)\{f\}\right|\\
&&=\left|\E\left(p(r+W^{\resm})g_{m,f}(W^{\resm}+1)-W^{\resm}g_{m,f}(W^{\resm})\ind_{W^{\resm}>m}\right)\right|\\
&&=\frac{\left|\E\B_0 L_f(W)\right|}{\Pr(W\ge m)}\le \frac{\epsilon_1 \|L_f\|_0 + \epsilon_2 \|\Delta L_f\|_0}{\Pr(W\ge m)}.
\end{eqnarray*}
However, $\|\Delta L_f\|_0=G_{m,2}\vee \sup_{f\in\cf_m}|\Delta L_f(m)|$, therefore it remains to show that $\sup_{f\in\cf_m}|\Delta L_f(m)|=\sup_{f\in\cf_m}|g_{m,f}(m+1)|\le G_{m,2}$. To this end, {similar to the derivation of \Ref{gdecomp}} we use the strong Markov property to obtain
\begin{align*}
g_{m,f}(m+1)&= \E\int_0^{\tau^{\resm}_{m,m+1}} [f(Z^{\resm}_{m}(t)) - \NB^{\resm}(r,p)({f})] dt \\\
	&= \frac{1}{p(r+m)} \left( f(m) - \NB^{\resm}(r,p)({f}) \right),
\end{align*}
which ensures $|g_{m,f}(m+1)|\le \frac{1-\Pr(Z^{\resm} = m)}{p(r+m)} {= G_{m,2}}$. 
\ignore{To maximise the above, we choose $f = \ind_{\{m\}}$,
\[ \Rightarrow \Delta g_m(m) = \frac{1-\Pr(Z^{\resm} = m)}{r(p+m)},\]
which corresponds exactly to the bound \Ref{deltag}. Utilising $g(m) = 0$, then,
\begin{align*}
\E[ r(p+& W^{\resm}) g(W^{\resm}+1) - W^{\resm}g(W^{\resm}) \ind_{W^{\resm} > m}]\\
	&= \E[ r(p+W^{\resm}) g(W^{\resm}+1) - W^{\resm}g(W^{\resm})]\\
	&= \frac{ \E[ p(r+W)g(W+1) - Wg(W)]}{\Pr(W \geq m)},
\end{align*}
where we also define $g(i) = 0,$ $\forall i \leq m$. Therefore, as the numerator of the above is simply the term one bounds in the unconditional case, if \Ref{unconditional} is true, then we can use the exact same bound but using $G_{m,1}, G_{m,2}$ in place of their unconditional equivalents.}
\end{proof}

\rem{Some care is needed when we apply Lemma~\ref{uncond-cond-thm} because of the small probability $\Pr(W\ge m)$ in the denominator. %This is not unexpected because, although $W$ is built on events with small probability, $W^\resm$ is not. For example, if $W=\sum_{i=1}^nX_i$, where $X_i$'s  are i.i.d. Bernoulli random variables with small success probability $p$, the conditional probability   $\Pr(X_i=1|W\ge m)$ is approximately $m/n$. 
However, in reality, the most interesting case is for $m=1$ and we will show that the error bounds for approximations with $m=1$ are usually small.}

\subsection{Applications}

In our first example, we consider the exceedances of a sequence of independent but not necessarily identically distributed random variables. 

\begin{example} \label{example1} Let $X_i$, $i\in\{1,\dots,n\}$ be independent random variables with $\Pr(X_i>s)=p_i$ and $N_{s,n}$ be the number of exceedances above $s$, then
\begin{align}
d_{TV}(\mathcal{L}(N_{s,n}^\resone),\Po^\resone(\lambda)) \leq \frac{\frac{1-e^{-\lambda} - \lambda e^{-\lambda}}{\lambda (1-e^{-\lambda})} \sum_{i=1}^n p_i^2}{1-\prod_{i=1}^n(1-p_i)},
\end{align}
where $\lambda = \sum_{i=1}^np_i$.

In the case where $n$ is fixed, $p_i = p$ for all $i$ and $p$ is small, then the bound is asymptotically $\frac{p}{2}$.
\end{example}

\begin{proof} The claim easily follows from
Lemma~\ref{uncond-cond-thm}, Corollary~\ref{poissonbounds} and equation~(1.23) from Barbour, Holst \& Janson~(1992).
\end{proof}

\begin{example} Let $\{Y_i,\ 1\le i\le n\}$ be i.i.d. random variables with $p=\Pr(Y_i>s)$ and $X_i=Y_i\wedge Y_{i+1}$, $1\le i\le n$, where $Y_{n+1}:=Y_1$. 
With $b = \frac{2p-3p^2}{1+2p-3p^2}$ and $a = (1-b)np^2$, the number $N_{s,n}$ of exceedances of $\{X_i,\ 1\le i\le n\}$ above $s$ satisfies
\begin{align}
d_{TV}(\mathcal{L}(N_{s,n}^\resone),\NB^\resone(a/b,b)) \leq  \frac{32.2p}{\sqrt{(n-1)(1-p)^3}} \cdot \frac{a G_{1,2}}{\Pr(N_{s,n} \geq 1)},\label{tworuns}
\end{align}
with $G_{1,2}$ defined in \Ref{deltag}. When $n$ is fixed and $p$ is small, the upper bound is asymptotically 
$\frac{16.1p}{\sqrt{(n-1)(1-p)^3}}$.\end{example}

%\rem{For fixed $n$ and small $p$, the upper bound in \Ref{tworuns} is approximately half of the upper bound of $d_{TV}(\mathcal{L}(N_{s,n}),\NB(a/b,b))$ in Theorem 4.2 of Brown \& Xia~(2001).}

\begin{proof} Using Lemma~\ref{uncond-cond-thm}, one can exploit the proof of 
Theorem 4.2 in Brown \& Xia~(2001) word for word, with Stein's factors replaced by their conditional equivalents, to get the first claim.
For the asymptotic result, Theorem~8.G in Barbour, Holst \& Janson~(1992) gives
\begin{align*}
\Pr(N_{s,n} \geq 1) &\geq 1- e^{-\L(1-p)} - (5p^2 - 4p^3),
\end{align*}
which, after some elementary expansion, yields $\frac{a G_{1,2}}{\Pr(N_{s,n} \geq 1)}\asymp\frac12$.
\end{proof}

{Lemma}~\ref{uncond-cond-thm} essentially states that if a random variable $W$ of interest can be well approximated by a negative binomial random variable, then its conditional distribution can {also} be well approximated by the conditional negative binomial distribution. The following example shows that the converse is not true. In other words, the conditional negative binomial approximation may be appropriate even if the unconditional approximation is poor.

Consider a sequence of random variables $X_1, \ldots, X_n$ which are conditionally independent given a parameter random variable $\Theta$, where $\Theta$ takes values 0 or 1 with distribution $\Pr(\Theta=1)=1-\Pr(\Theta=0)=q$. When the parameter $\Theta=0$, exceedances will not happen, while $\Theta=1$ is the phase where exceedances may happen. Such phenomena are very common in quality control, seismology and  finance. In quality control (Lambert~(1992)), if manufacturing equipment is well maintained, {it} will not fail, but when the equipment is wearing out {due to} insufficient maintenance, it has a positive probability to fail during its operation. In seismology (Ellsworth \& Beroza~(1995)), earthquakes are strongly linked to a distinctive seismic nucleation phase. In finance (Yalamova \& McKelvey~(2011)), the ``herding behavior" is often linked to different phases of the financial market with some dying off and some leading to crashes. These phenomena can not be modelled by a Poisson distribution as the errors of approximation are too large to justify a Poisson approximation. For this reason, a zero-inflated Poisson is a more suitable choice (Lambert~(1992)).

\begin{example} With the setup in the preceding paragraph, let $p_1=\Pr(X_1>s|\Theta=1)$ and $\L=np_1$, then the number $N_{s,n}$ of exceedances satisfies
%$$d_{TV}(\cL(N_{s,n}^{\resm}),\Po^{\resm}(\L))\le \frac{G_{m,2}np_1^2}{1-\sum_{i=0}^{m-1}{n\choose i}p_1^i(1-p_1)^{n-i}},$$
%where $G_{m,2}$ is given in Corollary~\ref{poissonbounds}. In particular, we have
$$d_{TV}(\cL(N_{s,n}^\resone),\Po^\resone(\L))\le \frac{p_1(1-e^{-\L}-\L e^{-\L})}{1-(1-p_1)^n}$$
and the bound is asymptotically $\frac 12 p_1$ for fixed $n$ and small $p_1$.
\end{example}

\begin{proof} For convenience, we write $W=N_{s,n}$ and we wish to approximate $W^{\resone}$ with $\Po^{\resone}(\L)$. To this end, we observe from \Ref{steineqPo} that
\begin{align*}
&\left|\E f(W^{\resone})-\Po^\resone(\L)\{f\}\right| =\left|\E[\C_1g_{1,f}(W^{\resone})]\right|\\
&= \left|\frac{\E[\C_1g_{1,f}(W)\ind_{W\ge 1}]}{\Pr(W\geq 1)}\right|=\left|\frac{\E[\C_1g_{1,f}(W)\ind_{W\ge 1}|\Theta=1]}{\Pr(W\geq 1|\Theta=1)}\right|\\
&\le \frac{G_{1,2}np_1^2}{1-(1-p_1)^n},
\end{align*}
where the inequality is derived as in the proof of Example~\ref{example1}. 
\end{proof}

%%%%%%%%%%%%%%%%%%%%%%%%%%
\def\rsa{{\it Rand.\ Struct.\ Alg.\/}~}
\def\ac{{\it Academic Press, New York\/}}
\def\aap{{\it Adv.\ Appl.\ Prob.\/}~}
\def\ap{{\it Ann.\ Probab.\/}~}
\def\anst{{\it Ann.\ Statist.\/}~}
\def\amst{{\it Ann.\ Math.\ Statist.\/}~}
\def\anap{{\it Ann.\ Appl.\ Probab.\/}~}
\def\jap{{\it J.~Appl.\ Probab.\/}~}
\def\jws{{Wiley, New York}}
\def\ny{{New York}}
\def\ptrf{{\it Probab.\ Theory Rel.\ Fields\/}~}
\def\sp{{Springer}~}
\def\spa{{\it Stoch.\ Procs.\ Appl.\/}~}
\def\sv{{Springer-Verlag, Berlin}}
\def\tpa{{\it Theor.\ Probab.\ Appl.\/}~}
\def\zw{{\it Z.~Wahrsch.\ verw.\ Geb.\/}~}
\def\berksym{{\it Proc.\ Sixth Berkeley Symp.\ Math.\ Statist.\ Prob.\/}~}
\def\spl{{\it Statist.\ Probab.\ Lett.\/}~}
\def\lns{{Lecture Notes in Statistics}~}
\def\statsci{{\it Statist.\ Science\/}~}
\def\mcap{{\it Meth.\ Comput.\ Appl.\ Probab.\/}~}
\def\esaim{{\it ESAIM:\ Prob.\ \&~Statist.\/}~}
\def\astin{{\it ASTIN Bulletin\/}~}
\def\bullisi{{\it Bull.\ Int.\ Statist.\ Inst.\/}~}
\def\eljp{{\it Electron.\ J.~Probab.\/}~}
\def\cpc{{\it Combin.\ Probab.\ Comp.\/}~}
\def\pcps{{\it Proc.\ Cam.\ Philos.\ Soc.\/}~}
\def\tcs{{\it Theor.\ Comp.\ Sci.\/}~}
\def\sjad{{\it SIAM J.~Algebraic Discr.\ Methods\/}~}
\def\jma{{\it J.~Math.\ Anal.\/}~}
\def\sp{{Springer}}

\end{document}